\documentclass[a4paper,12pt,reqno]{amsart}
\usepackage{fullpage}
\usepackage{mymacros}
\usepackage{oitp}
\usepackage[all]{xy}
\usepackage{stmaryrd,mdwlist}
\usepackage{hyperref}
\usepackage{breakurl}
\title[Derived equivalence and Grothendieck ring of varieties]{Derived equivalence and Grothendieck ring of varieties: the case of K3 surfaces of degree 12 and abelian varieties}

\date{}
\author[A.~Ito]{Atsushi Ito}
\address{Graduate School of Mathematics,
Nagoya University,
Furocho, Chikusaku, Nagoya, 464-8602, Japan}
\email{atsushi.ito@math.nagoya-u.ac.jp}

\author[M.~Miura]{Makoto Miura}
\address{
Korea Institute for Advanced Study,
85 Hoegiro,
Dongdaemun-gu,
Seoul,
130-722,
Republic of Korea.
}
\email{miura@kias.re.kr}

\author{Shinnosuke Okawa}
\address{
Department of Mathematics,
Graduate School of Science,
Osaka University,
Machikaneyama 1-1,
Toyonaka,
Osaka,
560-0043,
Japan.
}
\email{okawa@math.sci.osaka-u.ac.jp}

\author[K.~Ueda]{Kazushi Ueda}
\address{
Graduate School of Mathematical Sciences,
The University of Tokyo,
3-8-1 Komaba,
Meguro-ku,
Tokyo,
153-8914,
Japan.}
\email{kazushi@ms.u-tokyo.ac.jp}

\begin{document}

\maketitle

\begin{abstract}
In this paper,
we discuss the problem
of whether
the difference $[X]-[Y]$
of the classes
of a Fourier--Mukai pair $(X, Y)$
of smooth projective varieties
in the Grothendieck ring of varieties
is annihilated by some power
of the class $\mathbb{L} = [ \mathbb{A}^1 ]$
of the affine line.
We give an affirmative answer for Fourier--Mukai pairs of very general K3 surfaces of degree 12.
On the other hand,
we prove that in each dimension greater than one,
there exists an abelian variety such that
the difference with its dual is not annihilated by any power of $\mathbb{L}$,
thereby giving a negative answer to the problem.
We also discuss
variations of the problem.

%We discuss the following question: 

%We show that general K3 surfaces of degree $12$ come
%in non-isomorphic Fourier--Mukai pairs $(X, Y)$ 
%satisfying $[\bA^3] \cdot ([X]-[Y]) = 0$
%in the Grothendieck ring of varieties.
%We also show that for each integer $g \ge 2$ the Fourier--Mukai pair of
%a very general abelian $g$-fold and its dual have different classes in a completion.
%In particular, the difference of their classes is not annihilated by a power of
%$
% \bL = [ \bA ^{ 1 } ]
%$
%in the Grothendieck ring of varieties.
%By the same arguments, we also show that the Hilbert schemes of points of a Fourier--Mukai pair of
%abelian surfaces is another such example.
\end{abstract}

\tableofcontents

\section{Introduction}
Let
$
 X
$
and
$
 Y
$
be a pair of smooth projective varieties
(or more generally smooth and proper Deligne--Mumford stacks)
over a field $\bfk$.
We say that $X$ is \emph{D-equivalent} to $Y$
if the bounded derived category of coherent sheaves
$
 D ( X ) \coloneqq D ^{ b } \coh X
$
is equivalent to
$
 D ( Y )
$
as a $\bfk$-linear triangulated category
(we also say that $Y$ is a \emph{Fourier--Mukai partner} of $X$).

It is shown in the pioneering paper \cite{Mukai_DAPS} that
an abelian variety is D-equivalent to its dual, meaning that non-birational varieties could be D-equivalent.
The following natural question arises from this observation.

\begin{question}\label{q:basic_question}
Which piece of information of a variety
(or more generally Deligne--Mumford stack)
does the derived category have?
In other words,
if $X$ and $Y$ are D-equivalent, which invariants of $X$ and $Y$ do coincide?
\end{question}

It follows from the uniqueness of the Serre functor
\cite{MR1039961}
that $X$ and $Y$ have the same dimension and isomorphic (anti-)canonical rings.
%since the graded piece
%of the (anti-)canonical ring
%is the group of natural transformations
%from the identity functor
%to a multiple of the Serre functor
%shifted by the dimension.
%(and hence the same (anti-)Kodaira dimension)
%\cite{9506012,MR1818984,
%MR1949787}.
%Derived invariance of Hochschild (co)homologies are also well known,
%and goes back at least to 
%\cite{MR1035222,MR1099084}
%in the context of finite-dimensional algebras.
%Hochschild cohomology of schemes
%
%Loday \cite{MR925871} gave a definition
%of the cyclic (co)homology of schemes
%by sheafifying the cyclic (co)homologies of algebras
%and taking the hypercohomology.
%This definition is generalized to Hochschild (co)homologies
%in \cite{MR981619,MR1120653,MR1390671}.
%
The Hochschild cohomology ring,
whose graded pieces are
the group of natural transformations
from the identity functor to the shift functor,
is also derived invariant.
When the characteristic of the base field $\bfk$ is zero, combined with the Hochschild--Kostant--Rosenberg isomorphism
\cite{MR0142598,MR1390671,MR1940241,MR2141853,MR2955193},
this gives the following partial coincidence of the Hodge numbers for any integer $i \in \bZ$;
\begin{align} \label{eq:vertical_sum}
 \sum _{ q - p = i } h ^{ p, q } ( X ) = \sum _{ q - p = i } h ^{ p, q } ( Y ).
\end{align}
%The reader can refer to, for example, \cite{MR2244106} for these results.
Another deep result is the coincidence of Chow motives with rational coefficients up to Tate twists
%which can be shown as an application of the theory of noncommutative motives
\cite{MR2225203, MR2196100, MR2641191, MR3108695}.
%\cite{MR3281141}.

On the other hand, Popa and Schnell \cite{MR2839458} proved in characteristic zero
that the identity components of the Picard schemes of $X$ and $Y$ are isogenous.
As a corollary, it follows that the Hodge numbers of $X$ and $Y$
coincide for all $(p, q)$ under the assumption that $\dim X ( =\dim Y ) \le 3$.
Honigs \cite{MR3750214} also showed that, if $\bfk = \bF _{ q }$,
the Hasse--Weil zeta functions coincide again in dimension 3 or less.
%The aim of this paper, in a word, is to discuss simultaneous refinements of these results.

Recall that smooth projective varieties $X$ and $Y$ are \emph{K-equivalent}
if there is a diagram
\begin{align}
\begin{gathered}
 \xymatrix{ & W \ar[dl] _{ p } \ar[dr] ^{ q } &\\
X & & Y}
\end{gathered}
\end{align}
where $W$ is a normal variety and
$p, q$ are birational projective morphisms satisfying
$
 p ^{ * } K _{ X } \simeq q ^{ * } K _{ Y }.
$
It is conjectured and has been shown in some cases that K-equivalent varieties
(and stacks) are D-equivalent.
Conversely, it is also conjectured that birational and D-equivalent smooth projective varieties of non-negative
Kodaira dimension are K-equivalent. % \cite[Conjecture 1.2]{Kawamata_DEKE}\footnote{The assumption on the Kodaira dimension is necessary \cite{MR2067481}.}.
For the details and the history of DK hypothesis, see the recent survey article \cite{MR3838122} and references therein.

Combined with the arguments above, it is natural to ask if K-equivalent varieties have the same Hodge numbers
when $\bfk$ has characteristic zero. In fact,
this is known to be the case in its full generality. One way to prove this is to utilize the method of
\emph{motivic integration}, and here the \emph{Grothendieck ring of varieties} comes into the picture.

The \emph{Grothendieck ring of varieties} over a field
$
 \bfk
$,
which will be denoted by
$
 K _{ 0 } \lb \Var / \bfk \rb
$
in this paper, is the quotient of the free abelian group generated by the set of
isomorphism classes of schemes of finite type over
$
 \bfk
$
modulo the relations
\begin{equation}
 [ X ] = [ X \setminus Z ] + [ Z ]
\end{equation}
for closed embeddings
$
 Z \subset X.
$
Multiplication in
$
 K _{ 0 } \lb \Var / \bfk \rb
$
is defined by the Cartesian product,
which is easily seen to be associative, commutative, and unital with
$
 1 = [ \Spec \bfk ].
$

When $\bfk$ is of characteristic zero, it is shown by means of the motivic integration of Kontsevich \cite{MR1672108,MR1664700} that classes of smooth projective K-equivalent varieties coincide in the \emph{completed Grothendieck ring of varieties}.
This is a completion of the \emph{localized Grothendieck ring}
$
 K _{ 0 } \lb \Var / \bfk \rb [ \bL ^{ - 1 } ]
$,
where
$
 \bL = [ \bA ^{ 1 } ]
$
is the class of the affine line (see \pref{df:Grothendieck_ring}). From this result, by taking (an extension of)
the Hodge--Deligne polynomial, one immediately obtains the coincidence of
the Hodge numbers for K-equivalent varieties.

In view of this result, it is natural to ask
if D-equivalent varieties have the same class
in the completed (or localized) Grothendieck ring of varieties.

\begin{problem}[{\cite[Problem 1.3]{1612.08497v1}}]
\label{pb:naive}
Let
$
 \lb X, Y \rb
$
be a Fourier--Mukai pair. Does the equality
\begin{align} \label{eq:annihilate}
 ([X] - [Y]) \cdot \bL^k = 0 \in K _{ 0 } \lb \Var / \bfk \rb
\end{align}
hold for a non-negative integer $k$?
\end{problem}
\noindent For example, one can take $k=0$ if $X$ and $Y$ are
related by an elementary flop appearing in \cite{Bondal-Orlov_semiorthogonal}.

Recently a number of positive results
have been obtained by several groups of people.

\begin{enumerate}
\item
Borisov \cite{Borisov:2014aa}
show that the Pfaffian--Grassmannian pairs $(X, Y)$ of Calabi--Yau 3-folds
\cite{MR1775415}
satisfy
\begin{align} \label{eq:Borisov}
 \lb [X] - [Y] \rb (\bL^2-1)(\bL-1) \bL^7 = 0
\end{align}
in
$
 K _{ 0 } \lb \Var / \bfk \rb
$,
giving a first counter-example to % the cut-and-paste problem and
the cancellation problem,
which asks the injectivity of the homomorphism
$
 K _{ 0 } \lb \Var / \bfk \rb \to K _{ 0 } \lb \Var / \bfk \rb [ \bL ^{ - 1 } ]
$.
\pref{eq:Borisov} is subsequently refined by Martin \cite{MR3535349} to
\begin{equation} \label{eq:Martin}
 \lb [ X ] - [ Y ] \rb \cdot \bL ^{ 6 } = 0.
\end{equation}

\item
It is shown in \cite{MR3912058} that for pairs $(X', Y')$ of smooth Calabi--Yau 3-folds
obtained as certain degenerations of $X$ and $Y$,
one has
\begin{align}
 \lb [X'] - [Y'] \rb \cdot \bL = 0.
\end{align}
\suspend{enumerate}
Both the Pfaffian--Grassmannian pairs $(X, Y)$ of Calabi--Yau 3-folds
and their degenerations $(X', Y')$ are Fourier--Mukai partners
by \cite{MR2475813,0610957} and \cite{1611.08386}
(see also \cite{MR3959275}).
\resume{enumerate}

\item
In \cite{1612.07193v1}, Kuznetsov and Shinder studied the Fourier--Mukai partners of
K3 surfaces of degree $8$ and $2$.
A K3 surface $X$ is said to be of degree $d$
if there is an ample line bundle $L$ on $X$
satisfying $L \cdot L = d$.
They prove that there are such pairs
$
 \lb X '', Y '' \rb
$
satisfying
\begin{align}
 \lb [X ''] - [Y ''] \rb \cdot \bL = 0.
\end{align}

\item
In \cite{Borisov:2017aa, Ottem:2017aa},
it is shown that a generic pair of
the so-called $\text{GPK}^3$ Calabi--Yau 3-folds
$
 \lb X ''', Y ''' \rb
$
are D-equivalent, non-birational, and satisfy the equality
\begin{align}
 \lb [X '''] - [Y '''] \rb \cdot \bL ^{ 4 } = 0.
\end{align}
See also \cite{2017arXiv171110231K} for further study in this direction.

\item
Examples of pairs of K3 surfaces of degree 2 are given in \cite[Proposition 4.1]{2017arXiv171206958K}.

\item
A series of examples consisting of hyperk\"ahler manifolds of $K3[n]$-type is given in \cite{2018arXiv180109385O} (see also \cite{Meachan_2019} for finer results on birational (in)equivalence of those manifolds).

\item
Examples of Calabi--Yau 5-folds are given in \cite{MR3936623}.

\item
Further examples are given in \cite{2019arXiv190701335S}.
\end{enumerate}

\pref{pb:naive} was stated as \cite[Conjecture 1.6]{1612.07193v1}
around the same time as \cite{1612.08497v1}, the first preprint version of this paper.

One of the purposes of this paper is,
as an application of the geometry of equivariant vector bundles on homogeneous spaces
of type $D$, to prove that a general Fourier--Mukai pair of K3 surfaces of degree 12
gives another example
of an affirmative answer to \pref{pb:naive}.
%The coarse moduli space of polarized K3 surfaces of degree $d$
%is a quasi-projective variety.

\begin{theorem} \label{th:main}
Let $X$ be a general K3 surface of degree 12 over $\bC$.
Then there exists a non-isomorphic Fourier--Mukai partner $Y$ of $X$
satisfying
\begin{align}
 \lb [X] - [Y] \rb \cdot \bL^3 = 0.
\end{align}
\end{theorem}

\begin{remark}
After the completion of the first draft of this paper,
Brendan Hassett and Kuan-Wen Lai proved in \cite[Theorem 4.1]{Hassett:2016aa}
the stronger equality
\begin{align}
 \lb [X] - [Y] \rb \cdot \bL = 0 \in K _{ 0 } \lb \Var / \bC \rb,
\end{align}
by a completely different method based on Cremona transformations of $\bP^4$.
\end{remark}

It follows from
\cite[Proposition 1.10]{MR1892313}
(see also \cite[Corollary 2.7.4]{MR2096145})
that the number of isomorphism classes of Fourier--Mukai partners of a K3 surface over $\bC$
with Picard number 1 and degree 12 is 2.
Hence \pref{th:main} and \cite[Theorem 4.1]{Hassett:2016aa} gives an affirmative answer to \pref{pb:naive} for very general K3 surfaces of degree 12 over $\bC$.

%\begin{remark} \label{rm:KS}
%After the completion of the first draft of this paper,
%Alexander Kuznetsov and Evgeny Shinder posted the paper \cite{1612.07193} on the arXiv,
%which also discusses the relation between derived equivalences
%and annihilators of the class of the affine line.
%In particular, they also asked \pref{pb:naive} as \cite[Conjecture 1.6]{1612.07193}.
%\end{remark}

Our second purpose, however, is to explain that there are examples which negatively answer to \pref{pb:naive}.

\begin{theorem}%[$=$\pref{cr:counter-example}]
 \label{th:DLAV}
If
$
 A
$
is a complex abelian variety which is not isomorphic to its dual
$
 \Ahat
$, then
$
 [ A ] \neq [ \Ahat ]
$
in the completed Grothendieck ring of varieties
$\Khat _{ 0 } \lb \Var / \bC \rb$.
In particular, for any integer $g \ge 2$ there exists a pair of non-isomorphic complex abelian $g$-folds which are D-equivalent but have distinct classes in $\Khat _{ 0 } \lb \Var / \bC \rb$.
\end{theorem}

%The definition of the completed Grothendieck ring of varieties
%is recalled in \pref{df:Grothendieck_ring}.
\noindent The same example is also discovered independently in \cite[Theorem 3.1]{1707.08997}.
Our proof uses \pref{pr:Ekedahl2} ($=$\cite[Proposition 3.6]{Ekedahl:2009aa}),
which is based on the solution of Tate isogeny theorem by Faltings. In \cite{1707.08997}, this part is replaced with more elementary categorical arguments.

In \pref{cr:nonDL2}, we explain that the pair of Hilbert schemes of points
$
 \lb A ^{ [ n ] }, B ^{ [ n ] } \rb
$,
where
$
 \lb A, B \rb
$
is a pair of abelian surfaces as in \pref{th:DLAV}, is another example of a negative answer to \pref{pb:naive}

Because of these negative examples, Kuznetsov and Shinder modified their conjecture
in the published version
\cite[Conjecture 1.6]{MR3848025}
by adding the extra assumption that $X$ and $Y$ should be simply connected.
However, it is desirable to modify \pref{pb:naive} in such a way that it still has meaningful implications such as
the coincidence of the Hodge numbers over a field of characteristic zero and that of the number of rational points over a finite field, without adding extra assumptions on the pair
$
 ( X, Y )
$.
In \pref{sc:modifications},
we raise problems in this direction.

Another interesting problem is to extend the whole picture
in such a way that not only equivalences but also admissible embeddings of triangulated categories are
taken into account.
Let $\Gamma_\bfk$ be the ring
defined in \cite[Definition 8.1]{MR2051435}
as the quotient of the free abelian group
generated by quasi-equivalence classes
of enhanced (=pretriangulated dg)
bounded derived categories $\scrD ( X )$
of smooth complex projective varieties $X$
by the relations
\begin{align}
 [\scrD ( X )] = [\scrD ( Y_1 )] + \cdots + [\scrD ( Y_n )]
\end{align}
for semiorthogonal decompositions
\begin{align}
 \scrD ( X ) = \la \scrB_1, \ldots, \scrB_n \ra
\end{align}
with $\scrD ( Y_i ) \simeq \scrB_i$ for $i=1, \ldots, n$.
Multiplication in $\Gamma_\bfk$ is defined by the tensor product
of dg categories.
It is shown in \cite[Section 8]{MR2051435}
that there exists a motivic measure
\begin{align}
 K_0( \Var / \bfk ) \to \Gamma_\bfk
\end{align}
sending the class $[X]$ of a smooth projective variety
to $[\scrD ( X )]$,
which descends to a ring homomorphism
\begin{align} \label{eq:BLL}
 K_0( \Var / \bfk )/(\bL-1) \to \Gamma_\bfk.
\end{align}
%We will give an example in \pref{sc:K3_surfaces_in_G26}.
\pref{pb:naive} is closely related to the problem,
given implicitly at the end of \cite[Section 8]{MR2051435},
of asking how close to being injective
the map \pref{eq:BLL} is.
%\cite[Conjecture 1]{MR2225203} also has a flavor similar to \pref{pb:naive},
%although there is no apparent implication in either direction.

This paper is organized as follows:
In \pref{sc:OG48},
we give pairs $(X, Y)$ of Calabi--Yau manifolds
defined as zeros of sections of equivariant vector bundles
on homogeneous spaces of $\Spin(2 m)$,
and prove the equality \eqref{eq:XYL} in the Grothendieck ring of varieties.
By specializing to $m=4$,
we obtain pairs of K3 surfaces of degree 12.
In \pref{sc:K3_in_OG(5,10)},
we show that the K3 surfaces can be also described
as linear sections of a homogeneous space $\OG(5,10)$
of $\Spin(10)$.
In \pref{sc:PD},
we use results of Mukai
to complete the proof of \pref{th:main}.
%In \pref{rm:quadric_K3},
%we give pairs $(X, Y)$ of K3 surfaces
%in 5-dimensional quadrics
%satisfying $([X']-[Y']) \cdot \bL^2 = 0$.
%This is a side result
%which is not on the main line of discussion in this paper.
In \pref{sc:K3_surfaces_in_G26},
we give a formula in the Grothendieck ring of varieties which describes the class of Pfaffian cubic 4-folds by the class of the associated K3 surfaces and $\bL$. This is compatible with the homomorphism \eqref{eq:BLL},
in that the terms in the formula are sent to the SOD summands.
In \pref{sc:AV},
we discuss negative answers to \pref{pb:naive}.
In \pref{sc:modifications}, we discuss variations of \pref{pb:naive}.
%and a tentative modification of \pref{pb:naive}.

\subsection*{Note added}
After the completion of \pref{sc:modifications} and before it was made public, the preprint \cite{2019arXiv191004733B} appeared on the arXiv, which has some overlap with \pref{sc:modifications}.
In particular, our \pref{pr:conjecture in characteristic p implies conjecture in characteristic 0} is essentially the same as \cite[Theorem 1.6]{2019arXiv191004733B}.
Compared to us, the author of \cite{2019arXiv191004733B} seems more concentrated on the relationship between \(D\)-equivalence and Chow motives (see \cite[Propositions 1.7, 1.8]{2019arXiv191004733B}).

\subsection*{Notation and conventions}
Throughout the paper, we use the dual of the Grothendieck's convention for projectivization. Namely, for a locally free sheaf \( \cF \) on a scheme \(X\), we write \( \bP _{ X } ( \cF )\) to mean \( \Proj _{ X } \Sym \cF ^{ \vee }\). In particular, for \( X = \Spec \bfk \) and \( V = \cF \), the \( \bfk \)-rational points of the scheme \( \bP ( V ) = G( 1, V ) \) represent lines of \( V \).

\subsection*{Acknowledgements}
We thank Kenji Hashimoto and Daisuke Inoue for collaboration at an early stage of this work;
this note is originally conceived as a joint project with them.
We also thank
%Alexander Kuznetsov for informing the authors of \cite{1612.07193}, %\pref{rm:KS},
Genki Ouchi for \pref{rm:Ouchi} and the reference
\cite{MR2051435},
and
Kota Yoshioka for the reference \cite{MR1987738}.
We also thank Yujiro Kawamata, Keiji Oguiso,
Evgeny Shinder, Hokuto Uehara, and Takehiko Yasuda for useful discussions.
We also thank the anonymous referees for reading the manuscript carefully
and suggesting a number of improvements.
A.~I.~was supported by
Grants-in-Aid for Scientific Research (14J01881,17K14162).
M.~M.~was supported by Korea Institute for Advanced Study.
S.~O.~was partially supported by Grants-in-Aid for Scientific Research
(16H05994,
16K13746,
16H02141,
16K13743,
16K13755,
16H06337,
18H01120)
and the Inamori Foundation.
K.~U.~was partially supported by Grants-in-Aid for Scientific Research
(24740043,
15KT0105,
16K13743,
16H03930).

\section{K3 surfaces in $\OG(4,8)$}
 \label{sc:OG48}

Let $V_0$ be a vector space of dimension $m$, so that
$
 \Vnat \coloneqq V_0 \oplus V_0^\dual
$
is a vector space of dimension
$
 2m
$
equipped with the natural non-degenerate bilinear form
\begin{align}
 \pair{v+\check{v}}{v'+\check{v}'}
  = \check{v}(v') + \check{v}'(v).
\end{align}
The \emph{spin group} $G \coloneqq G _{ m } \coloneqq \Spin(\Vnat)$
is the simply-connected simple algebraic group,
which is obtained as the double cover of $\SO(\Vnat)$.
The \emph{spinor representation}
is a $2^m$-dimensional representation of $G$
on $\bigwedge V_0$,
which decomposes as the direct sum
of \emph{half spinor representations}
$
 V_1 \coloneqq \bigwedge^\even V_0
$
and
$
 V_2 \coloneqq \bigwedge^\odd V_0.
$
These half spinor representations are related to each other
by an outer automorphism of $\Spin(2m)$.
They are self-dual if $m$ is even,
and dual to each other if $m$ is odd.
Let $\omega_1$ and $\omega_2$
be the fundamental weights corresponding to
the half spinor representations $V_1$ and $V_2$ respectively.
Given a dominant integral weight $\lambda$ of $G$,
the irreducible representation of $G$
with highest weight $\lambda$
will be denoted by $V_\lambda$.
We also write $(i,j) \coloneqq i \omega_1 + j \omega_2$,
so that
$
 V_1 = V_{(1,0)}
$
and
$
 V_2 = V_{(0,1)}.
$

Let $F_i$ be the homogeneous space of $G$
associated with the Dynkin diagram
\begin{align}\label{eq:DynkinD}
  \DynkinD
\end{align}
of type $D_m$ with the $i$-th node crossed out
(see \cite[Section 2.3]{MR1038279} for the notion and rudiments of crossed Dynkin diagrams).
The nodes of the Dynkin diagram correspond to the simple roots,
 and the fundamental weights $\omega_1, \dots, \omega_m$.
 In
 this paper we use the numbering of nodes as in \eqref{eq:DynkinD}.
For $i = 3, \ldots, m$,
the numbering fits well with the description of \( F _{ i } \) as orthogonal Grassmannians:
\begin{align}
 F_i
  &= \OG(m-i+1, \Vnat) \nonumber \\
 \label{eq:og}
  &\coloneqq \lc V \subset \Vnat \relmid \pair{-}{-}|_V = 0 \text{ and } \dim V = m-i+1 \rc.
\end{align}
We also let $\OG(m,\Vnat)$ and $\OG(m,2m)$
denote one of the connected components of the space \pref{eq:og} for $i=1$,
which coincides with the spinor variety $F_1\simeq F_2$.
Let further
$
 F_{ 1, 2 }
$
be the homogeneous space of $G$
associated with the Dynkin diagram
\eqref{eq:DynkinD}
with the nodes 1, 2 crossed out.
It can naturally be identified with
$
 \OG(m-1, \Vnat)
  \simeq \bP_{F_i}(\cS_i^\dual)
$
for $i=1, 2$,
where $\cS_i$ is the tautological subbundle of
\( \Vnat \otimes _{ \bfk } \cO _{ F _{ i } } \)
on $F_i$. %$\OG(m, \Vnat)$.
\begin{comment}
Note that for any isotropic subspace
$V \subset \Vnat$ of dimension $m-1$,
there are exactly two isotropic subspaces
$\Vtilde \subset \Vnat$ of dimension $m$.
One of them is contained in $F_1$,
and the other is contained in $F_2$.
These subspaces correspond to
$\OG(1, V^\bot/V) \simeq \OG(1,2)$.
This gives natural maps
$\OG(m-1, \Vnat) \to F_i$ for $i=1, 2$,
whose fibers over $\Vtilde \in \OG(m, \Vnat)$
is $\bP(\Vtilde^\dual)$.
\end{comment}
They fit into the following diagram:
\begin{equation} \label{eq:diagram_F12}
\begin{gathered}
\xymatrix{
 & F_{ 1, 2 } \ar[ld]_{p_1} \ar[rd]^{p_2} & \\
 F_1& & F_2
}
\end{gathered}
\end{equation}

The Picard group of
$
 F _{ 1, 2 }
$
is given by
\begin{equation}
 \Pic \lb F _{ 1, 2 } \rb
 =
 p _{ 1 } ^{ * } \Pic \lb F _{ 1 } \rb
 \oplus
 p _{ 2 } ^{ * } \Pic \lb F _{ 2 } \rb
 \simeq \bZ^2.
\end{equation}
For
$
 i = 1, 2
$,
we let
$
 \cO_{F_i}(1)
$
denote the ample generator of
$
 \Pic \lb F _{ i } \rb
$.
Then the line bundle
$
 \cO_{F_{ 1, 2 }}(1,1) \coloneqq \cO_{F_1}(1) \boxtimes \cO_{F_2}(1)
$
is very ample, and one has
\begin{align}
 F _{ 1, 2 } \simeq \bP_{F_i} \lb \lb p _{ i * } \cO_{F_{ 1, 2 }}(1,1) \rb^\dual \rb
\end{align}
for both $i=1, 2$.
Recall that homogeneous vector bundles on
$
 F _{ i }
$
correspond to representations of the parabolic subgroup
$
 P_i \subset G
$
corresponding to
$
 F _{ i }
$,
which in turn correspond to representations
%of the commutator subgroup, i.e.\ the semi-simple part, 
of the Levi subgroup of
$
 P _{ i }
$ if they are irreducible representations.
In our case the Levi subgroup is isomorphic to
$
S(\GL(m) \times \GL(1))
$,
where
$
 S ( \ )
$
denotes the subgroup consisting of elements of determinant $1$.
%in $GL(m+1)$.
In this sense
$
 p _{ i * } \cO_{F_{ 1, 2 }}(1,1)
$
is the locally free sheaf of rank $m$ on $F_i$
associated with the representation of
$
 S(\GL(m) \times \GL(1))
$
with highest weight $(1,1)$.
%Since
%$
% \cO_{F_{ 1, 2 }}(1,1)
%$
%is globally generated, so are
%$
% p _{ i * } \cO_{F_{ 1, 2 }}(1,1)
%$
%by
%\cite[Proposition 1.8]{MR1201388}.
If $i$ and $j$ are non-negative integers,
then one has
\begin{align}
 H^0 \lb \cO_{F_{ 1, 2 }}(i,j) \rb
  \simeq V_{(i,j)}^\dual
\end{align}
by the Borel--Weil theorem.

Let
\begin{equation}
 \label{eq:section}
 s \in H ^{ 0 } \lb F _{ 1, 2 }, \cO_{F_{ 1, 2 }}(1,1) \rb
 \simeq
 H ^{ 0 } \lb F _{ 1 }, p _{ 1 * } \cO_{F_{ 1, 2 }}(1,1) \rb
 \simeq
 H ^{ 0 } \lb F _{ 2 }, p _{ 2 * } \cO_{F_{ 1, 2 }}(1,1) \rb
\end{equation}
be a general section and let
\begin{equation}\label{eq:definition_of_X_and_Y}
\begin{split}
 D \coloneqq Z ( s ) \subset F _{ 1, 2 },\\
 X \coloneqq Z ( p _{ 1 * } s ) \subset F _{ 1 },\\
 Y \coloneqq Z ( p _{ 2 * } s ) \subset F _{ 2 },
\end{split}
\end{equation}
be its zero loci,
which are smooth complete intersections by Bertini
\cite[Theorem 1.10]{MR1201388}.
One can easily compute the rank and the degree of ${p_i}_* \cO_{F_{ 1, 2 }}(1,1)$
to show that $X$ and $Y$ are Calabi--Yau of dimension $m(m-3)/2$.
Set
$
 \pi _{ i } = p _{ i } | _{ D } \colon D \to F _{ i }
$
for
$
 i = 1, 2
$.

\begin{lemma} \label{lm:D}
The morphism
$
 \pi_1
$
is a $\bP^{m-2}$-bundle over
$
 F_1 \setminus X
$
and a $\bP^{m-1}$-bundle over
$
 X,
$
which are locally trivial in the Zariski topology.
The same holds for
$
 \pi _{ 2 }
$.
\end{lemma}

\begin{proof}
Since the arguments for
$
 \pi _{ 1 }
$
and
$
 \pi _{ 2 }
$
are the same, we only give it for
$
 \pi _{ 1 }
$.
Fix a point
$
 x \in F _{ 1 }
$.
Since
$
 p _{ 1 }
$
is a projective bundle and
$
 \cO_{F_{ 1, 2 }}(1,1)
$
is
$
 p _{ 1 }
$-ample, by cohomology-and-base-change,
we obtain the standard isomorphism
\begin{equation}
 p _{ 1 * } \cO_{F_{ 1, 2 }}(1,1) | _{ x }
 \simto
 H ^{ 0 } \lb p _{ 1 } ^{ - 1 } ( x ), \cO_{F_{ 1, 2 }}(1,1) | _{ p _{ 1 } ^{ - 1 } ( x ) } \rb
\end{equation}
sending
$
 \lb p _{ 1 * } s \rb ( x )
 \in
 p _{ 1 * } \cO_{F_{ 1, 2 }}(1,1) | _{ x }
$
to
$
 s | _{ p _{ 1 } ^{ - 1 } ( x ) }
$.
Hence
\begin{equation}
 D \cap p _{ 1 } ^{ - 1 } ( x )
 =
 \pi _{ 1 } ^{ - 1 } ( x )
\end{equation}
is isomorphic to
$
 p_1^{ - 1 } ( x ) \simeq \bP^{m-1}
$
if $x \in X$, and to a hyperplane therein
otherwise.
It follows that the short exact sequence
\begin{align}
 0 \to \cO _{ F _{ 1 } } \stackrel{{p_1}_*s}{\to} p _{ 1 * } \cO_{F_{ 1, 2 }}(1,1) \to p _{ 1 * } \cO_{F_{ 1, 2 }}(1,1) / \cO _{ F _{ 1 } } =: \cQ \to 0
\end{align}
splits Zariski locally on
$
 F _{ 1 } \setminus X
$,
so that
$
 \pi _{ 1 }
$
is the
$
 \bP ^{ m - 2 }
$
bundle over
$
 F _{ 1 } \setminus X
$
associated to the locally free sheaf
$
 \cQ
$.
The latter claim follows from
$
 s | _{ p _{ 1 } ^{ - 1 } ( X ) } = 0
$.
\end{proof}

\begin{corollary} \label{cr:main}
One has
\begin{align} \label{eq:XYL}
 ([X]-[Y]) \bL^{m-1} = 0
\end{align}
in the Grothendieck ring of varieties.
\end{corollary}

\begin{proof}
One has
\begin{align}
 [D]
  = ([F_1]-[X])[\bP^{m-2}] + [X][\bP^{m-1}]
  = ([F_2]-[Y])[\bP^{m-2}] + [Y][\bP^{m-1}]
\end{align}
by \pref{lm:D}.
Since
$F_1$ and $F_2$ are related by an outer automorphism of $\Spin(2m)$
and hence isomorphic as algebraic varieties,
one has
\begin{align}
 0
  = ([X] - [Y])([\bP^{m-1}]-[\bP^{m-2}])
  = ([X] - [Y]) \bL^{m-1},
\end{align}
and \pref{cr:main} is proved.
\end{proof}

\begin{remark}
 \label{rm:ambient}
Let $\G(m,\Vnat)$ be the Grassmannian of $m$-spaces in $\Vnat = V_0 \oplus V_0^\dual$,
and $\cS$ be the universal subbundle on it.
The zero of the section $s'$ of 
$
 \Sym^2 \cS^\dual
$
associated with the natural pairing $\pair{-}{-}$ on $\Vnat$
is the homogeneous space of the form
\begin{align}
 \Or \lb \Vnat \rb / \GL(V_0).
\end{align}
It has two connected components,
which can be identified with $F_1$ and $F_2$
 in such a way that
 $\Vnat=V_{\omega_m}$ as a $G$-module,
 %By the (iso)morphism 
 %$H^0(\cS^\dual)\to H^0\lb \cS^\dual|_{F_i}\rb$ of $G$-modules and
 %the Borel--Weil theorem,
 and $\scS^\dual|_{F_i}$ is
the equivariant vector bundle
associated with the irreducible representation of $P_i$
with the highest weight $\omega_m$ for both $i=1$ and 2.
Note that the lowest weight of this irreducible representation of $P_1$ (resp. $P_2$)
is $\omega_1-\omega_2$ (resp.~$-\omega_1+\omega_2$),
so that $\cS|_{F_1}$ (resp.~$\cS|_{F_2}$) is the equivariant vector bundle
associated with the irreducible representation of $P_1$ (resp.~$P_2$)
with the highest weight $-\omega_1+\omega_2$ (resp.~$\omega_1-\omega_2$).
Since the ample generator of $\Pic \G(m, \Vnat)$ restricts
to twice the ample generator of $\Pic F_1$,
one has
\begin{align}
 \cS(1)|_{F_1}
  &\simeq {p_1}_* \cO_{F_{ 1, 2 }}(-1,1) \otimes {p_1}_* \cO_{F_{ 1, 2 }}(2,0) \\
  &\simeq {p_1}_* \cO_{F_{ 1, 2 }}(1,1).
\end{align}
The same reasoning for $F_2$ gives
\begin{align}
 \cS(1)|_{F_2}
  ={p_2}_* \cO_{F_{ 1, 2 }}(1,1).
\end{align}
Since $H^0(\cS(1)) \to H^0(F_1, \cS(1)|_{F_1})$
is a projection to an irreducible component of restricted representation
by the Borel--Weil theorem, 
the section $s$ in 
\pref{eq:section}
is extended as an element of $H^0(\cS(1))$.
Therefore, it follows that the zero of the section
\begin{align}
 s + s'
  \in H^0 \lb %\G(m, V_0 \oplus V_0^\dual),
   \cS(1) \oplus \Sym^2 \cS^\dual \rb
\end{align}
is isomorphic to the disjoint union of $X$ and $Y$.
\end{remark}

\begin{remark}
 \label{rm:triality}
When $m=4$,
the homogeneous spaces $F_1$ and $F_2$ are related to $F_4$
by an outer automorphism of $\Spin(8)$
called the \emph{triality automorphism}.
The homogeneous space
$
 F_4 = \OG(1,\Vnat)
$
is a quadric hypersurface in $\G(1,\Vnat) \simeq \bP^7$,
and $X$ and $Y$ are K3 surfaces of degree 12.
\end{remark}

%
%
%%%%%%%%%%%%%%%%%%%%%%%%%%%%%%%%%%%%%%%%%%%%%%

\section{K3 surfaces in $\OG(5,10)$}
\label{sc:K3_in_OG(5,10)}
We use the same notation as in \pref{sc:OG48}.
We write
$
 \cO_{F_1}(i) \coloneqq p _{ 1 * } \cO _{ F _{ 1, 2 } } \lb i, 0 \rb
$
and
$
 \cE \coloneqq p _{ 1 * } \cO _{ F _{ 1, 2 } } \lb 0, 1 \rb,
$
so that
$
 H^0(\cO_{F_1}(1)) \simeq V_1^\dual 
$
and
$
 H^0(\cE) \simeq V_2^\dual.
$
Let
\begin{equation} \label{eq:diagram_H12}
\begin{gathered}
\xymatrix{
 & H_{ 1, 2 } \ar[ld]_{h_1} \ar[rd]^{h_2} & \\
 H_1 & & H_2
}
\end{gathered}
\end{equation}
be the diagram
obtained by replacing $m$ with $m+1$ in the diagram \pref{eq:diagram_F12}.
 We use the same numbering $1, 2, \dots, m+1$ for the nodes of 
 the Dynkin diagram $D_{m+1}$ as that of $D_m$ in \pref{sc:OG48}.
In some places we will write
\( G _{ m + 1 } \)
to mean the spin group of type \( D _{ m + 1 } \), to avoid confusion.
\pref{lm:tits} below
is an adaptation of \cite[Lemma 4.2]{Inoue2019}.

\begin{lemma} \label{lm:tits}
Let $F_1 \subset H_1$ be the equivariant embedding corresponding to
the unique inclusion of the Dynkin diagram of type $D_m$ with the node 1 crossed out into that of type $D_{m+1}$ with the same node crossed out.

\begin{enumerate}
\item The following morphism is surjective.
\begin{align}
 \left. h_2 \right|_{h_1^{-1}(F_1)} \colon h_1^{-1}(F_1) \to H_2
\end{align}
\item
There exists an isomorphism of projective bundles over $F_1$ as follows.
 \begin{equation}\label{eq:projbdl}
 \begin{gathered}
 \xymatrix{
  h_1^{-1}(F_1) \ar[rd]_{\left. h_1 \right|_{h_1^{-1}(F_1)}} \ar[rr] ^-{ \simeq } &
  & \bP \coloneqq \bP_{F_1} \lb \cO_{F_1}(-1) \oplus \cE^\dual \rb \ar[ld]^{\pi}\\
  &  F_1  &
 }
 \end{gathered}
 \end{equation}
 \end{enumerate}
\end{lemma}

\begin{proof}
 The argument is almost the same as the proof of \cite[Lemma 4.2]{Inoue2019},
 except for some computations relevant to Tits transformations
 and equivariant vector budles. 
 Hence we will be brief.

 Consider the following diagram, which corresponds to \cite[(75)]{Inoue2019}.
 \begin{equation}\label{eq:diagram}
 \begin{gathered}
 \xymatrix{
  & H_{1,2}
  \ar[ld]_{h_2} \ar[rd]^{h_1} 
  &	& H_{1,m+1}
  \ar[ld]_{h_1'} \ar[rd]^{h_{m+1}} &  \\
  H_2
  & h_1^{-1}(F_1) = \bP_{F_1}\lb \cG|_{F_1}\rb
  \ar[ld] \ar@{}[u]|{\bigcup} \ar[rd]
  & H_1
  & 
  F_1\ar@{=}[ld] \ar@{}[u]|{\bigcup} \ar[rd]
  & H_{m+1}
  \\
  \cT(F_1)
  \ar@{}[u]|{\bigcup} & &
  F_1\ar@{}[u]|{\bigcup} & & o 
  \ar@{}[u]|{\rotatebox{90}{$\in$}}\\
 }
 \end{gathered}
 \end{equation}
 This diagram is obtained in the following way.
 Fix a Borel subgroup \( B \) of $G _{ m + 1 } =\Spin(2m+2)$, and
 let $B\subset P_i \subset G _{ m + 1 }$ be the maximal parabolic subgroups corresponding to the fundamental weight $\omega_i$ for $i=1,\dots,m+1$.
 We will write $H_i \coloneqq G _{ m + 1 }/P_i$ and $H_{i,j}\coloneqq G _{ m + 1 }/\lb P_i \cap P_j\rb$, consistently with the notation of \pref{eq:diagram_H12}.
 Let $W$ and $W_{P_i}$ denote 
 the Weyl groups of $G$ and $P_i$, respectively.
 Let further $W^{P_i} \subset W$ be the set of minimal length representatives
 of the coset $W_{P_i}\backslash W$ for $i=1, \ldots, m+1$.

 To show the assertion (1),
 we apply \cite[Lemma 2.4]{MR3130568} twice to the diagram \pref{eq:diagram};
 first to the Tits transform $F_1=\cT(o)\coloneqq h_1'\lb h_{m+1}^{-1}(o)\rb$ 
 of the point $o\coloneqq P_{m+1}/P_{m+1}$, which is induced by the right hut in the diagram,
 and then to the Tits transform $\cT(F_1)\coloneqq h_2\lb h_1^{-1}(F_1)\rb$ 
 induced by the left hut.
 Each of the two transformations is a correspondence of Schubert cycles, hence is
 described as the correspondence of (representatives of) the appropriate cosets of the Weyl groups.
Concretely, the correspondences for the first and the second transformations are given by the composition of the following maps for $(i,j)=(m+1, 1)$ and $(1,2)$ respectively,
where
$w_i$ denotes
the longest elements of $W_{P_i}$ for all $i$.
 \begin{align}
  t_{i,j}\colon  W^{P_i} \to w_i W^{P_i} \subset W \to W_{P_j}
  \backslash W \simeq W^{P_j}
 \end{align}
 Together with \pref{lm:comb} below, 
 we have 
 \begin{align}
  t_{m+1,1}(\id)=\widetilde{w}_{m+1}
  \text{ and } t_{1,2}(\widetilde{w}_{m+1})=w^{P_2},
 \end{align}
 where $\wtilde_{m+1} \in W^{P_1}$ be the minimal length representative
 of $W_{P_1}w_{m+1} \in W_{P_1}\backslash W$,
 and $w^{P_i}$ denotes
 the longest element of $W^{P_i}$ for all $i$.
 Therefore $F_1$ and $\cT(F_1)$ turn out to be the Schubert varieties
 \[
  \overline{B \widetilde{w}_{m+1}^{-1} P_1/P_1}
  \subset H_1 \text{ and }
  X(w^{P_2}) \coloneqq \overline{B (w^{P_2})^{-1} P_2/P_2} \subset H_2,
 \]
 respectively.
 Since $\dim X(w^{P_2}) = l(w^{P_2})$ is the maximum, we obtain $\cT(F_1)= X(w^{P_2}) = H_2$.

 For (2),
 we consider two equivariant vector bundles $\cG^\dual$  over $H_1$ 
 and $\cG'^\dual$ over $H_{1,m+1}$,
 both of which correspond to the irreducible representations
 of the highest weight $\omega_2$,
 i.e., 
 $H^0\lb H_1, \cG^\dual \rb \simeq H^0\lb H_{1,m+1}, \cG'^\dual \rb 
 \simeq V_2^\dual$
 in the sense of \(G _{ m + 1 }\)-modules.
 Since $h_1^{-1}(F_1)=\bP_{F_1}\lb \cG|_{F_1}\rb$, it is enough to show 
 $\cG|_{F_1} = \cO_{F_1}(-1)\oplus \cE^\dual$.
 As we can see in the same way as \cite[(77)]{Inoue2019},
 we have a short exact sequence
 \begin{align}
  0 \to \cG'|_{F_1} \to \lb h_1'^*\cG\rb|_{F_1} \simeq \cG | _{ F _{ 1 } }
  \to \cO_{F_1}(-1) \to 0.
 \end{align}
 Note that there is an equivalence of categories
 \begin{align}
  \mathrm{VB}_G(G/P) \simeq \module(P); \quad \cN \mapsto \cN|_{P/P} %_{[ 1 _{ G } \cdot P ]}
 \end{align}
 between the category of $G$-equivariant vector bundles on $G/P$
 and the category of $P$-modules, for any parabolic subgroup $P\subset G$.
  Under this equivalence, 
 $\cG'^\dual|_{F_1}$ corresponds to the $(P_1\cap P_{m+1})$-module of highest weight $\omega_2$
 regarded as a \(\lb P_1\cap \Spin(2m) \rb \)-module.
 Since this is the \(\lb P_1\cap \Spin(2m) \rb \)-module of highest weight $\omega_2$,
 we have $\cG'|_{F_1} \simeq \cE^\dual$.
 Moreover, $\cE^\dual(1)$ corresponds to the irreducible \(\lb P_1\cap \Spin(2m) \rb \)-module
 whose highest weight is the dominant weight $\omega_{m+1}$. 
 Hence, by the Bott--Borel--Weil theorem,
 we obtain
 \[
  \Ext^1(\cO_{F_1}(-1),\cE^\dual) = H^1(\cE^\dual(1)) =0.
 \]
 Thus we conclude the proof.
\end{proof}

\begin{lemma} \label{lm:comb}
 Let us use the same notation in the proof of \pref{lm:tits}. It holds
 \begin{align}
  W_{P_2}w_1 \wtilde_{m+1} = W_{P_2} w^{P_2}.
 \end{align}
\end{lemma}
\begin{proof}
    One of a reduced decomposition of the longest element of the
	Weyl group of each type is given by \cite{MR1628449}.
	In particular, we have
	\begin{align}
			w_1&= (s_2)(s_3s_2)(s_4s_3s_2)(\dots)(s_{m+1}\dots s_4s_3s_2),\\
			\label{eq:wmplus1}
			w_{m+1}&=(s_1s_2)(s_3s_1s_2s_3)(s_4s_3s_1s_2s_3s_4)(\dots)(s_m
			\dots s_4s_3s_1s_2s_3s_4 \dots s_m),
	\end{align}
	where $s_1, \dots, s_{m+1} \in W$ are simple reflections. 
	By using \pref{eq:wmplus1} and by computing 
	the right action of $w_{m+1}\in W_{P_{m+1}}$ on a weight $\omega_1$,
	we can derive the following formula inductively.
	\begin{align}
			\wtilde_{m+1}=\begin{cases}
			(s_1)(s_3s_2)(s_4s_3s_1)(\dots)(s_m\dots s_4s_3s_1) & \text{for even $m$,}\\
			(s_1)(s_3s_2)(s_4s_3s_1)(\dots)(s_m\dots s_4s_3s_2) & \text{for odd $m$.}
			\end{cases}	
	\end{align}
	Similarly, we can also check by an induction
	that the right action of $w_1 \wtilde_{m+1} \in W$ on
	a weight $\omega_2$ is the same as that of 
	\begin{align}
	 (s_2s_3s_4\dots s_{m+1})\wtilde_{m+1} \in W^{P_2},
	\end{align}
	which is nothing but the longest element $w^{P_2}\in W^{P_2}$ as the lengths coincide. 
\end{proof}

Consider the embedding $F_1 \subset G(m,\Vnat)$ in \pref{rm:ambient}.
For each maximal isotropic subspace $V \subset \Vnat$,
the inner product $\pair{-}{-}$ gives a canonical linear isormorphism $\Vnat/V \simeq V^\dual$.
Hence we have a short exact sequence
\begin{align}
 \label{eq:seq}
 0 \to \cS|_{F_1} \to \Vnat \otimes \cO_{F_1} 
 \to \cS^\dual|_{F_1} \to 0
\end{align}
on $F_1$, where $\cS$ is the universal subbundle on $G(m,\Vnat)$.
 Recall that $\cS^\dual|_{F_1}$ is 
 the equivariant vector bundle associated with the irreducible
 representation of $P_1$ with the highest weight $\omega_m$, 
 so that
 $H^0(S^\dual|_{F_1})\simeq \Vnat^{\dual} = \Vnat$ as $\Spin(2m)$-modules.

From now, we consider only the case
$
 m = 4
$.
By the triality for the Dynkin diagram of type $D_4$, 
 we have another short exact sequence on $F_1$,
\begin{align}
 0 \to \cE^\dual \to V_2^\dual \otimes \cO_{F_1} 
 \to \cE \to 0.
\end{align}

Consider the morphism
%in addition to \pref{eq:projbdl},
\begin{align}
 \mu \colon \bP \coloneqq \bP_{F_1} \lb \cO_{F_1}(-1) \oplus \cE^\dual \rb
 \to \bP \lb V_1 \oplus V_2^\dual \rb
\end{align}
defined by the natural inclusion
\begin{align}
 \cO_{F_1}(-1) \oplus \cE^\dual
  \subset \lb V_1 \oplus V_2^\dual \rb \otimes \cO _{ F _{ 1 } }.
\end{align}
Let
$
\Sigma
  \coloneqq \mu(\bP)
  \subset \bP \lb V_1 \oplus V_2^\dual \rb
$
be the image of \( \mu \) with the reduced structure.

\begin{lemma}
%The embedding
%\( \Sigma \hookrightarrow \bP \lb V_1 \oplus V_2^\dual \rb \)
%is projectively normal.

There exists an isomorphism
 \( H _{ 2 } \simto \Sigma \)
 which fits in the following commutative diagram.
\begin{align}
\label{eq:commutative}
\begin{CD}
 h_1^{-1}(F_1) @>{\sim}>> \bP \\
 @V{\left. h_2 \right|_{h_1^{-1}(F_1)}}VV @VV{\mu}V \\
 H_2 @>{\sim}>> \Sigma
\end{CD}
\end{align}
\end{lemma}

\begin{proof}
We write
\( \bP ' \coloneqq h_1^{-1}(F_1) \)
for simplicity.
Since \( \dim \bP = \dim H_2 \), it follows from \cite[Lemma 3.13]{MR3130568} that the  morphism $h _{ 2 } | _{ \bP ' }$ is a birational morphism. 
Since \( H _{ 2 } \) is smooth and hence normal, it follows that
\( \left(h _{ 2 } | _{ \bP ' }\right) _{ * } \cO _{ \bP ' } \simeq \cO _{ H _{ 2 }} \), so that
\(
  H ^{ 0 } \left( \bP ', \left(h _{ 2 } | _{ \bP ' }\right) ^{ * } \cO _{ H _{ 2 } } ( 1 )\right)
  \simeq
  H ^{ 0 } \left( H _{ 2 }, \cO _{ H _{ 2 } } ( 1 ) \right)
\).
Therefore, the composition of the morphism $h _{ 2 } | _{ \bP ' }$ with the embedding
\(
  H _{ 2 } \hookrightarrow \bP H ^{ 0 } \left( H _{ 2 }, \cO _{ H _{ 2 } } ( 1 ) \right) ^{ \dual }
\)
coincides with the morphism defined by the complete linear system associated to the line bundle
\( \left(h _{ 2 } | _{ \bP ' }\right) ^{ * } \cO _{ H _{ 2 } } ( 1 ) \).

Also, the isomorphism
\(
  H ^{ 0 } \left( \bP, \cO _{ \pi } ( 1 ) \right)
  \simeq
  V _{ 1 } ^{ \dual } \oplus V _{ 2 }
\)
implies that the morphism \( \mu \) is associated to the complete linear system of the line bundle
\( \cO _{ \pi } ( 1 ) \).
Therefore it is enough to show that the line bundle
\( \left(h _{ 2 } | _{ \bP ' }\right) ^{ * } \cO _{ H _{ 2 } } ( 1 ) \)
on \( \bP ' \)
corresponds to
\( \cO _{ \pi } ( 1 ) \)
on
\( \bP \)
under the isomorphism \( \bP ' \simeq \bP \). Note, by the commutative diagram \eqref{eq:projbdl},
this is equivalent to the isomorphism
\(
  \cO _{ h _{ 1 } | _{ \bP ' } } ( 1 )
  \simeq
  \left(h _{ 2 } | _{ \bP ' }\right) ^{ * } \cO _{ H _{ 2 } } ( 1 )
\).

To see this, consider the projective bundle
\(
  h_1 \colon H_{1,2}=\bP_{H_1}(\cG) \to H_1
\).
The isomorphism
\(
  \cG^\dual \simeq {h_1}_* h_2^* \cO _{ H _{ 2 } } ( 1 )
\)
implies that
$
 \cO_{h_1} (1)
 \simeq
 h_2^* \cO _{ H _{ 2 } } ( 1 )
$.
%the morphism $h_2 : H_{1,2} =\bP_{H_1}(\cG) \to H_2$ is induced by the complete linear system associated to $\cO_{h_1} (1)$ for $h_1 : \bP_{H_1}(\cG) \to H_1$.
%Thus $h_2^* \cO_{H_2} (1) \simeq \cO_{h_1} (1)$.
On the other hand, since the projective bundle
\(
  h_1|_{\bP'} \colon \bP'= \bP_{F_1}(\cG|_{F_1}) \to F_1
\)
is obtained as the base change of
\( h _{ 1 } \)
by
\(
  F _{ 1 } \hookrightarrow H _{ 1 }
\),
we obtain the sequence of isomorphisms
\(
  \cO_{h_1|_{\bP'}} (1)
  \simeq
  \cO _{ h _{ 1 }} ( 1 ) | _{ \bP ' }
  \simeq
  h_2^* \cO _{ H _{ 2 } } ( 1 ) | _{ \bP ' }
  =
  \left(h _{ 2 } | _{ \bP ' }\right) ^{ * } \cO _{ H _{ 2 } } ( 1 )
\).
%Thus we conclude the proof.
\begin{comment}
Hence,
\( \left(h _{ 2 } | _{ \bP ' }\right) ^{ * } \cO _{ H _{ 2 } } ( 1 ) \simeq \cO_{h_1|_{\bP'}} (1) \)
on \( \bP ' \)
corresponds to
\( \cO _{ \pi } ( 1 ) \)
on
\( \bP \)
under the isomorphism $\cG|_{F_1} \simeq \cO_{F_1}(-1) \oplus \cE^\dual $.
\end{comment}
\end{proof}

Now we restrict ourselves to the case $m=4$,
where we have the following diagram:
\begin{equation}\label{eq:m4}
 \begin{gathered}
 \xymatrix{
   & H_{1,2}
  \ar[ld]_{h_2} \ar[rd]^{h_1} & \\
   H_2
  & \bP \ar[ld]_{\mu} \ar@{}[u]|{\bigcup} \ar[rd]^{\pi}
  & H_1 \\
   H_2
  \ar@{=}[u] & &
  F_1\ar@{}[u]|{\bigcup} 
 }
 \end{gathered}
\end{equation}

Let
$
 \sigma \in H^0(F_1, \cE(1))
$
be the image of a linear map
$
 \sigmatilde \colon V_1 \to V_2^\dual
$
by the map
\begin{align} \label{eq:s}
 V_1^\dual \otimes V_2^\dual
  \simeq H^0(\cO_{F_1}(1)) \otimes H^0(\cE)
  \to H^0(\cE(1)).
\end{align}
The map \pref{eq:s} is surjective,
since it is a non-zero $\Spin(8)$-equivariant map
and $H^0(\cE(1)) \simeq V_{(1,1)}^\dual$ is an irreducible
representation of $\Spin(8)$.
In particular,
$\sigma$ is general if so is $ \sigmatilde$.
On the other hand, 
the map $\sigmatilde$ induces a linear embedding
\begin{align} \label{eq:embedding}
 \bP \lb\Gamma_{\sigmatilde} \rb \colon \bP(V_1) \hookrightarrow \bP(V_1 \oplus V_2^\dual),
  \qquad [p] \mapsto [(p, \sigmatilde(p))],
\end{align}
whose image will be denoted by
$
 L_{\sigmatilde}.
 \subset \bP(V_1 \oplus V_2^\dual).
$
Conversely, if a 7-dimensional linear subvariety
$
 L \subset \bP(V_1 \oplus V_2^\dual)
$ satisfies $L \cap \bP(V_2^\dual) = \emptyset$,
then there exists an element  
$
 \sigmatilde \in V_1^\dual \otimes V_2^\dual
$
such that
$
 L = L_{\sigmatilde}.
$
It means that 
$L_{\sigmatilde}$ is general if so is $ \sigmatilde$.
Thus 
a general $\sigma \in H^0(\cE(1))$ corresponds to
a general $L_{\sigmatilde}\subset \bP(V_1\oplus V_2^\dual)$.

\pref{pr:Z=Z} below is an adaptation
of \cite[Proposition 4.1]{Inoue2019}
to the present situation.
The proof 
is identical, 
and hence omitted.

\begin{proposition} \label{pr:Z=Z}
The zero locus $
 Z(\sigma)
  \subset F_1
  \subset \bP(V_1)
 $ of $\sigma \in H^0\lb\cE(1)\rb$
and the linear section
$
 H_2 \cap L_\sigmatilde
  \subset L_\sigmatilde
  \simeq \bP(V_1)
$
are projectively equivalent.
\end{proposition}

One concludes that
a general linear section of $H_2=\OG(5,10)$
is projectively equivalent
(and hence isomorphic)
to the zero of a general section of $\cE(1)$ on $F_1=\OG(4,8)$, which is nothing but $X$ in \pref{sc:OG48} for $m=4$.

\section{Projective duality and derived equivalence}
 \label{sc:PD}

%Consider the case
%$
% m = 5
%$.
The orthogonal Grassmannian $\OG(5,10)$ of dimension 10
can be embedded into the projectivization $\bP$
of 16-dimensional half spinor representation of $\Spin(10)$.
The projective dual variety of $\OG(5,10)$
in the dual projective space $\bPv$
is also isomorphic to $\OG(5,10)$.
It is known in \cite[Theorem 0.3]{MR977768}
that a generic K3 surface of degree 12
can be described as the intersection $\OG(5,10) \cap L$
with a linear subspace $L \subset \bP^{15}$
of codimension 8,
which is unique up to the action of $\Spin(10)$.

Consider a pair of K3 surface $( X, Y )$ which is constructed as in 
\eqref{eq:definition_of_X_and_Y} for $m=4$, and realize $X$ as $\OG(5,10) \cap L$.
As explained in \cite[Example 1.3]{MR1714828},
the intersection
\begin{align}
 \Xv \coloneqq \OG(5,10) \cap L^\bot \in \bPv^{15}
\end{align}
of the dual $\OG(5,10)$
with $L^\bot$ in the dual projective space
is the moduli space
\begin{align}
 \Xv \simeq \cM_X(2, \cO_X(1), 3),
\end{align}
where $\cM_X(r, \ell, t)$ is the moduli space
of stable sheaves $E$ on $X$ satisfying
\begin{align}
 \rank E = r, \ 
 c_1(E) = \ell, \ 
 \chi(E) = r + t.
\end{align}
$(X, \Xv)$ give Fourier--Mukai partners,
which are not isomorphic if $X$ has Picard number one by \cite[Theorem 2.1]{MR1987738} (see also its proof).
%(cf.~\cite{MR977768,MR1363081,MR1714828}).

\begin{proposition} \label{pr:Xv=Y}
One has an isomorphism $\Xv \simeq Y$.
\end{proposition}

\begin{proof}
We use the notation introduced in \pref{sc:K3_in_OG(5,10)}.
A general element
$
  \sigmatilde \in V_1^\dual \otimes V_2^\dual
$
defines the linear subspaces
\begin{align}
 L_1 \coloneqq \Image \lb (\id_{V_1},  \sigmatilde) \colon \bP(V_1) \to \bP(V_1 \oplus V_2^\dual) \rb
\end{align}
and
\begin{align}
 L_2 \coloneqq \Image \lb (-  \sigmatilde, \id_{V_2}) \colon \bP(V_2) \to \bP(V_1^\dual \oplus V_2) \rb,
\end{align}
in the dual projective spaces,
which are mutually orthogonal;
\begin{align}
 L_1 = (L_2)^\bot, \quad
 L_2 = (L_1)^\bot.
\end{align}
\pref{pr:Z=Z} shows
\begin{align}
 X = L_1 \cap H_2 \subset F_1, \quad
 Y = L_2 \cap H_1 \subset F_2,
\end{align}
which immediately implies \pref{pr:Xv=Y}.
\end{proof}

\pref{th:main}
follows from \pref{cr:main} for $m = 4$, the last line of \pref{sc:K3_in_OG(5,10)},
and \cite[Example 1.3]{MR1714828}.
%In fact, we seem to obtain the L-equivalence from the previous proposition on the projective duality. However, in order to obtain the exponent \( 3 \) in \pref{th:main}, we do need all these results.

\begin{remark}\label{rm:quadric_K3}
Let $\bO$ be the complexified Cayley octonion algebra.
%According to \cite[Theorem 2.6]{1605.04680},
Consider the projective space of null-square imaginary octonions
\begin{align}
 Q \coloneqq \lc \ld \sum_{i=0}^7 x_i \bfe_i \rd \in \bP(\bO) \relmid
  x_0 = \sum_{i=1}^7 x_i^2 = 0
  \rc,
\end{align}
and set
$
 Q _{ 1 } = Q _{ 2 } = Q
$
with coordinates $x$ and $y$, respectively. Together with
\begin{align}
 Q _{ 1, 2 } \coloneqq \lc (x, y) \in Q _{ 1 } \times Q _{ 2 } \relmid x \cdot y = 0 \in \bO \rc
\end{align}
they form the following diagram, where the projections $q_1$ and $q_2$ are both
$\bP^2$-bundles.
\begin{equation} \label{eq:diagram_Q12}
\begin{gathered}
\xymatrix{
 & Q_{ 1, 2 } \ar[ld]_{q_1} \ar[rd]^{q_2} & \\
 Q_1 & & Q_2,
}
\end{gathered}
\end{equation}

Let $s'$ be a general section of $H^0 \lb \cO_{Q_{ 1, 2 }}(1,1) \rb$ and set
\begin{align}
\begin{split}
 D' \coloneqq Z ( s' ) \subset Q _{ 1, 2 },\\
 X' \coloneqq Z ( q _{ 1 * } s' ) \subset Q _{ 1 },\\
 Y' \coloneqq Z ( q _{ 2 * } s' ) \subset Q _{ 2 }.
\end{split}
\end{align}
Then both $X'$ and $Y'$ are K3 surfaces of degree 12,
and the same reasoning as \pref{cr:main} gives the equality
\begin{align}
 ([X']-[Y']) \bL^2 = 0
\end{align}
in the Grothendieck ring of varieties. In fact, the pairs
$
 \lb X ', Y ' \rb
$
are degenerations of the pairs
$
 \lb X, Y \rb
$
discussed above.
It is an interesting problem to see if $X'$ and $Y'$ are not isomorphic,
and characterize K3 surfaces
which can be obtained in this way.
\end{remark}

\section{K3 surfaces in $\G(2,6)$}\label{sc:K3_surfaces_in_G26}

Let $W$ be a vector space of dimension 6.
The Pfaffian hypersurface
$
 \Pf(W) \subset \bP \lb \bigwedge^2 W^\dual \rb
$
consists of skew bilinear forms
%$
% \omega
%  \colon W \otimes W \to \bfk
%  \in \bigwedge^2 W^\dual
%$
on $W$
whose rank is strictly smaller than 6.
A general linear subspace
$
 L \subset \bP \lb \bigwedge^2 W^\dual \rb
$
of dimension 5 determines a cubic 4-fold
\begin{align}
 X \coloneqq \Pf(W) \cap L
\end{align}
and a K3 surface
\begin{align}
 Y \coloneqq \G(2,W) \cap L^\bot.
\end{align}
It is known by \cite[Theorem 2]{0610957}
that there is a fully faithful functor
$
 \iota \colon D ( Y ) \to D ( X )
$
and a semiorthogonal decomposition
\begin{align}\label{eq:sod_for_cubic_fourfold}
 D ( X )
  = \la \cO_X, \cO_X(1), \cO_X(2), \iota D ( Y ) \ra.
\end{align}
Fix a linear subspace $U \subset W$ of dimension 5.
The rational map
\begin{align}
 \varphi \colon \Pf(W) \dashrightarrow \bP(U)
\end{align}
induced by the composition of the wedge square
$
 (-)^{\wedge 2} \colon \bigwedge^2 W^\dual \to \bigwedge^4 W^\dual,
$
the projection
$
 \bigwedge^4 W^\dual \to \bigwedge^4 U^\dual,
$
and the (canonical up to scalar) isomorphism
$
 \bigwedge^4 U^\dual \simto U
$
sends
$
 x \in \bigwedge^2 W^\dual \subset \Hom(W, W^\dual)
$
to
$
 \ker x \cap U \subset U.
$
It is defined by a linear subsystem of $|2H|$,
where $H$ is the hyperplane section of $\Pf(W)$.
The base locus of this linear subsystem is the union
%$B_1 \cup B_2$
of
\begin{align}
 B_1 \coloneqq \lc x \in \Pf(W) \relmid \ker x \subset U \rc
\end{align}
and
\begin{align}
 B_2 \coloneqq \lc x \in \Pf(W) \relmid \dim \ker x = 4 \rc.
\end{align}
The inclusion
$
 U \subset W
$
induces a linear projection
$
 \pi_U \colon \bP \lb \bigwedge^2 W^\dual \rb
  \dashrightarrow \bP \lb \bigwedge^2 U^\dual \rb,
$
and
$
 B_1
$
is the closure of the inverse image of
\begin{align}
 \G(2, U)
  = \lc x \in \bP \lb \bigwedge^2 U^\dual \rb \relmid \dim \ker x = 3 \rc
  \subset \bP \lb \bigwedge^2 U^\dual \rb
\end{align}
by this rational map.
\begin{comment}
Note that
$
 \dim \bP \lb \bigwedge^2 W^\dual \rb = 14
$
and
$
 \dim \bP \lb \bigwedge^2 U^\dual \rb = 9.
$
For any $K \in \G(2, U)$,
the fiber over $K$ can be identified
with the projectivization
$
 \bP \lb \bigwedge^2 (\ker(W^\dual \to K^\dual)) \rb
$
of the space of anti-symmetric linear maps
from $W/K$ to its dual.
This is a 5-dimensional linear subspace of $\bP \lb \bigwedge^2 W^\dual \rb$.
It follows that
$
 \dim B_1 = \dim \G(2, U) + 5 = 11.
$
\end{comment}
If $L$ is general,
then $B_2 \cap L$ is empty,
and $S \coloneqq B_1 \cap L$ projects isomorphically by $\pi_U$
to the section
$
 \G(2,U) \cap \pi_U(L) \subset \bP \lb \bigwedge^2 U^\dual \rb
$
of $\G(2,U)$
by the linear subspace $\pi_U(L)$
%$
% L' \coloneqq L^\bot \cap \bP \lb \bigwedge^2 U \rb
%  \subset \bP \lb \bigwedge^2 W \rb
%$
%is a linear subspace of $\bP \lb \bigwedge^2 U \rb$
of dimension 5,
which is a del Pezzo surface of degree 5
embedded anti-canonically into $\pi_U(L) \simeq L$.
Let $\Xtilde \coloneqq \Bl_S X$ be the blow-up of $X$ along $S$,
whose exceptional divisor will be denoted by $E$.
It is known by \cite[Proposition 2]{MR749015} that
\begin{itemize}
 \item
the linear system $\left| 2 \Htilde - E \right|$ gives a birational morphism
$\varphitilde \colon \Xtilde \to \bP(U)$,
where $\Htilde$ is the total transform of the hyperplane section of $X$, and
 \item
the exceptional locus $F \subset \Xtilde$ of $\varphitilde$ is a smooth divisor,
whose image $G \subset \bP^4$ is isomorphic to the 5-point blow-up of $Y$.
\end{itemize}
The proof of \cite[Proposition 2]{MR749015},
which is based on \cite[Theorem 1]{MR571105},
actually shows that
the morphism
$
 \varphitilde
$
is the blow-up of
$
 \bP ^{ 4 }
$
along
$
 G
$.
Hence one has
\begin{align}\label{eq:motivic_formula_for_cubic_fourfold}
 \ld \Xtilde \rd
  &= [X] + \bL [S] \\
  &= [X] + \bL ( 1 + \bL + \bL^2 + 4 \bL)
\end{align}
and
\begin{align}
 \ld \Xtilde \rd
  &= \ld \bP^4 \rd + \bL [G] \\
  &= 1 + \bL + \bL^2 + \bL^3 + \bL^4 + \bL ([Y] + 5 \bL),
\end{align}
so that
\begin{align}\label{eq:decomposition_of_cubic_fourfold}
 [X] = 1 + \bL^2 + \bL^4 + \bL [Y].
\end{align}
This is compatible with the SOD
\eqref{eq:sod_for_cubic_fourfold} under the motivic measure \eqref{eq:BLL},
in the sense that the terms in the right hand side of \eqref{eq:motivic_formula_for_cubic_fourfold} are sent to the summands of \eqref{eq:sod_for_cubic_fourfold}.

\begin{remark} \label{rm:Ouchi}
We learned from Genki Ouchi that he independently obtained the formula
\eqref{eq:decomposition_of_cubic_fourfold} multiplied by 
$
 \bL^5(\bL-1)^2(\bL+1)
$
using an argument similar to \cite{Borisov:2014aa}.
See \cite[Section 2.6]{1612.07193v1} for similar results for other types of cubic fourfolds.
\end{remark}

\section{Abelian varieties}
 \label{sc:AV}

We prove \pref{th:DLAV} in this section.
First recall the following definition.

\begin{definition} \label{df:Grothendieck_ring}
%\begin{enumerate}
%
%\item
%Set $\bL = [ \bA ^{ 1 } _{ \bfk } ]$.
The \emph{localized Grothendieck ring of varieties}
$
 K _{ 0 } \lb \Var / \bfk \rb [ \bL ^{ - 1 } ]
$
is the localization
of the Grothendieck ring of varieties
$
 K _{ 0 } \lb \Var / \bfk \rb
$
by the class $\bL$ of the affine line.
%will be called the \emph{localized Grothendieck ring of varieties}.
%
%\item
For each integer $i \in \bZ$, let
$
 \Fil _{ i } \subset K _{ 0 } \lb \Var / \bfk \rb [ \bL ^{ - 1 } ]
$
be the abelian subgroup spanned by the elements of the form
$
 [ X ] \cdot \bL ^{ - m },
$
where $ m \in \bZ $ is an integer and $X$ is a variety such that
$
 \dim X - m \le i
$.
The subgroups $ \lb \Fil _{ i } \rb _{ i \in \bZ }$ form an ascending filtration of
$
 K _{ 0 } \lb \Var / \bfk \rb [ \bL ^{ - 1 } ]
$
satisfying
$
 \Fil _{ i } \cdot \Fil _{ j } \subset \Fil _{ i + j }
$.
The \emph{completed Grothendieck ring of varieties} is the completion with respect to this filtration:
\begin{align}\label{eq:completion_map}
 K _{ 0 } \lb \Var / \bfk \rb [ \bL ^{ - 1 } ] \to \varprojlim _{ i \in \bZ }
 \frac{ K _{ 0 } \lb \Var / \bfk \rb [ \bL ^{ - 1 } ] }{ \Fil _{ i } }
 \eqqcolon \Khat _{ 0 } \lb \Var / \bfk \rb.
\end{align}
%\end{enumerate}
\end{definition}
%In this section we show the following result.
%
%\begin{corollary}\label{cr:counter-example}
%For any integer $g \ge 2$, there exists a pair of non-isomorphic complex abelian $g$-folds $( A, B )$ which are D-equivalent but $[ A ] \neq [ B ] \in \Khat _{ 0 } \lb \Var / \bC \rb$.
%\end{corollary}

%As in \cite[Theorem 3.4(i)]{Ekedahl:2009aa},
Let $A _{ \bfk }$ be the group completion of the commutative monoid of isomorphism classes of algebraic group schemes over $\bfk$ whose connected components are abelian varieties and whose group of geometric connected components is a finitely generated group, where the binary operation $+$ of the monoid is defined by the direct sum;
$
 [ A ] + [ B ] \coloneqq [ A \oplus B ]
$.
When $\bfk = \bC$, by using the Bittner presentation of the Grothendieck ring \cite{MR2059227}
(which in turn is based on resolution of singularities \cite{MR0199184}
and weak factorization \cite{MR1896232}),
Ekedahl proved the following:
%\emph{Picard realization}:

\begin{theorem}[{\cite[Theorem 3.4]{Ekedahl:2009aa}, cf.~also \cite[Appendix A]{MR3490596}}]
 \label{th:Ekedahl}
There is a homomorphism of abelian groups
\begin{align}
 \Pic _{ \bC } \colon \Khat _{ 0 } \lb \Var / \bC \rb \to A _{ \bC }
\end{align}
sending the class $[ X ]$ of a smooth proper variety $X$ to the class
$
 \cl{ \Pic ( X ) } = \cl{ \Pic ^{ 0 } ( X ) } + \cl{ \NS ( X ) }.
$
\end{theorem}

Ekedahl also proved the following:

\begin{proposition}[{\cite[Proposition 3.6]{Ekedahl:2009aa}}]
 \label{pr:Ekedahl2}
Assume that $A$, $B$ and $C$ are abelian varieties over $\bC$
satisfying $A \oplus C \simeq B \oplus C$.
Then
$
 \Hom \lb A, B \rb
$
is a locally free right module of rank 1 over $R \coloneqq \End(A)$
and the natural morphism of abelian varieties
$
 \Hom \lb A, B \rb \otimes_R A \to B
$
is an isomorphism.
\end{proposition}

The proof of \cite[Proposition 3.6]{Ekedahl:2009aa} is based on the Tate's isogeny theorem for abelian varieties over a field which is finitely generated over its prime field. This is first shown in \cite{MR718935} over number fields, and is generalized later (see \cite[Chapter VI \S 3 Theorem 1 b)]{MR766568}. See also \cite[Chapter IV \S 1 Corollary 1.2]{MR766568} and its proof).

\begin{corollary} \label{cr:main_abelian}
Let $X$ and $Y$ be smooth projective varieties over $\bC$ such that
$
 \End ( \Pic ^{ 0 } ( X ) ) \simeq \bZ.
$
%is a hereditary algebra satisfying the cancellation property.
If
%$
% [ \Pic ^{ 0 } ( X ) ] = [ \Pic ^{ 0 } ( Y ) ]
%$
$
 [X] = [Y]
$
holds in
$
 \Khat _{ 0 } \lb \Var / \bC \rb
$, then $\Pic ^{ 0 } ( X )$ is isomorphic to $\Pic ^{ 0 } ( Y )$.
\end{corollary}

\begin{proof}
\pref{th:Ekedahl} gives the equality
$
 [ \Pic ( X ) ] = [ \Pic ( Y ) ]
$
in the group
$
 A _{ \bC }
$.
This is equivalent to saying that there exists a group scheme $G$ whose neutral component is an abelian variety
and the group of components is a finitely generated abelian group such that
\begin{align}
 \Pic ( X ) \times G \simeq \Pic ( Y ) \times G.
\end{align}
By taking the neutral components, we obtain the isomorphism
\begin{align}\label{eq:equality_in_A}
 \Pic ^{ 0 } ( X ) \times G ^{ 0 } \simeq \Pic ^{ 0 } ( Y ) \times G ^{ 0 }.
\end{align}
Now \pref{pr:Ekedahl2} gives an isomorphism
$
 \Pic ^{ 0 } ( X ) \simeq \Pic ^{ 0 } ( Y )
$,
since a locally free $\bZ$-module of rank 1
is unique up to isomorphism.
\end{proof}

\begin{remark}
A ring $R$ is said to have the \emph{cancellation property}
if an isomorphism of finitely generated one-sided modules
$
 A \oplus C \simeq B \oplus C
$
over $R$ implies $A \simeq B$ whenever $C$ is a projective module.
In \pref{cr:main_abelian},
it suffices to assume that $\End \lb \Pic^0(X) \rb$
is a hereditary ring with the cancellation property.
\end{remark}

%\begin{theorem}\label{th:main_abelian}
%Let $A$ be an abelian variety such that
%$
% \End \lb A \rb = \bZ
%$
%and
%$
% A \not\simeq \Ahat
%$,
%where
%$
% \Ahat
%$
%is the dual abelian variety of $A$. Then
%$
% \cl{A} \neq \cl{\Ahat} \in \Khat _{ 0 } \lb \Var / \bC \rb
%$.
%\end{theorem}

\begin{lemma} \label{lm:existence}
For any integer
$
 g \ge 2,
$
there exists an abelian $g$-fold
$
 A
$
such that
$
 \End \lb A \rb = \bZ
$
and
$
 A \not\simeq \Ahat \coloneqq \Pic ^{ 0 } \lb A \rb
$.
\end{lemma}

\begin{proof}%[Proof of \pref{lm:existence}]
The authors learned the following construction of $A$ from the answer by Bjorn Poonen to a question in MathOverflow \cite{Poonen}. For the sake of completeness, here we give an explanation with details (with a minor change).

Let $J$ be the Jacobian of a very general smooth projective curve of genus $g$.
As explained in \cite[Lecture 15]{Dolgachev_Milan}, by applying \cite[p.\,359 Lemma]{MR770932} for the case $d = 2$ and \cite[Theorem 11.5.1]{MR2062673} we obtain
$
 \End ( J ) = \bZ
$.
Note also that $J$ is isomorphic to its dual, since it is the Jacobian of a curve and
hence is a principally polarized abelian variety.
Let $G \subset J$ be a finite subgroup whose order $n$ is not a $g$-th power of an integer.
Then $A \coloneqq J / G$ is never isomorphic to its dual, since otherwise the composition
\begin{align}
 J \xrightarrow{q} A \simto \Ahat \xrightarrow{\qhat} \Jhat \simto J
\end{align}
gives an endomorphism of $J$ of degree $n ^{ 2 }$,
contradicting
$
 \End ( J ) = \bZ
$
and the choice of $n$.
The isogeny
$
 q \colon J \to A
$
induces the isomorphism
$
 \End ( A ) \otimes \bQ \simeq \End ( J ) \otimes \bQ = \bQ,
$
and hence one has
$
 \End ( A ) = \bZ
$.
\end{proof}

\pref{th:DLAV} is an immediate consequence of \pref{cr:main_abelian} and \pref{lm:existence}.
Note that the existence of one example
in \pref{lm:existence} implies that a very general abelian variety
also is an example.

We can also construct another class of examples of D-equivalent varieties
whose classes are distinct in
$
 \Khat _{ 0 } \lb \Var / \bC \rb
$.

\begin{corollary}
 \label{cr:nonDL2}
Let
$
 A
$
be a complex abelian surface such that
$
 A \not\simeq \Ahat
$
and
$
 \End \lb A \rb = \bZ
$.
Take a natural number
$
 n \ge 1
$
and consider the Hilbert schemes of $n$-points
$
 A ^{ [ n ] }
$
and
$
 \Ahat ^{ [ n ] }
$.
Then the pair
$
 \lb A ^{ [ n ] }, \Ahat ^{ [ n ] } \rb
$
gives a negative answer to \pref{pb:naive}.
\end{corollary}

\begin{proof}
By \cite[Proposition 8]{MR2353249}, the derived equivalence between
$
 A
$
and
$
 \Ahat
$
induces a derived equivalence between
$
 A ^{ [ n ] }
$
and
$
 \Ahat ^{ [ n ] }
$.
On the other hand, the summation maps
$
 s \colon A ^{ [ n ] } \to A
$
and
$
 \shat \colon \Ahat ^{ [ n ] } \to \Ahat
$
are the Albanese maps of
$
 A ^{ [ n ] }
$
and
$
 \Ahat ^{ [ n ] }
$,
respectively.
Hence it follows from \pref{cr:main_abelian} that
$
 \cl{A ^{ [ n ] }} \neq \cl{\Ahat ^{ [ n ] }} \in \Khat _{ 0 } \lb \Var / \bC \rb.
$
\end{proof}

\section{Variations of \pref{pb:naive}} \label{sc:modifications}

Now that we found conterexamples to \pref{pb:naive},
we discuss its possible modifications
in this section.
The guiding principle is to weaken the assertion of \pref{pb:naive} without losing meaningful implications such as the D-invariance of the Hodge numbers and the number of rational points over finite fields.

Before we discuss the modifications, we point out that
this problem is also interesting from the point of view of the mirror symmetry.
A pair $(X, \Xv)$ of Calabi--Yau $n$-folds is said to be a \emph{topological mirror pair}
if
\begin{align} \label{eq:tms}
 h^{p,q}(X) = h^{n-p,q}(\Xv)
\end{align}
for $0 \le p,q \le n$,
%Mirror symmetry exchanges the `horizontal' and `vertical' directions of the Hodge diamond,
%so that
so that if $(X,Y)$ and $(\Xv, \Yv)$ are Fourier--Mukai pairs,
and $(X, \Xv)$ and $(Y, \Yv)$ are topological mirror pairs,
then \pref{eq:tms}
together with \pref{eq:vertical_sum}
%of the `vertical sum', one obtains the equality
implies
\begin{align}
 \sum_{p+q=i} h^{p,q}(X)
  &= \sum_{p+q=i} h^{n-p,q}(\Xv) \\
  &= \sum_{p+q=i} h^{n-p,q}(\Yv) \\
  &= \sum_{p+q=i} h^{p,q}(Y)
\end{align}
%of the `horizontal sum' of Hodge numbers
for $0 \le i \le 2n$.
In other words,
mirror symmetry exchanges the horizontal (or Hodge filtration) direction
and the vertical (or weight) direction in the Hodge diamond,
so that if the mirrors of Fourier--Mukai partners are Fourier--Mukai partners,
then not only the vertical sums but also the horizontal sums will be preserved.
Moreover,
if $(X, Y)$ is a Fourier--Mukai pair,
and $(X, Z)$ and $(Y,Z)$ are topological mirror pairs,
then one has
\begin{align}
 h^{p,q}(X)
  = h^{n-p,q}(Z)
  = h^{p,q}(Y).
\end{align}
This is the case, for example,
for the Pfaffian--Grassmannian pair of Calabi--Yau 3-folds
\cite{MR1775415} mentioned in Introduction.

We also note that the coincidence of the Hasse--Weil zeta functions is stronger than the coincidence of the Hodge numbers in the following sense.

\begin{proposition}\label{pr:conjecture in characteristic p implies conjecture in characteristic 0}
  Assume in dimension $d$ and over finite fields that the D-equivalence implies the coincidence of the Hasse--Weil zeta functions.
  Then it follows in the same dimension $d$ and over \( \bC \) that the D-equivalence implies the coincidence of the Hodge numbers.
\end{proposition}

\begin{proof}
\begin{step}
We reduce the proof to the case where everything is defined over a number field $K$.
Choose projective embeddings of $X$ and $Y$ into a projective space $\bP _{ \bC }$, so that they are defined by systems of homogeneous polynomials with coefficients in $\bC$.
We choose such embeddings that
$
 H ^{ 1 } ( \bP _{ \bC }, I _{ X / \bP _{ \bC } } ( i ) ) = 0 = H ^{ 1 } ( \bP _{ \bC }, I _{ Y / \bP _{ \bC } } ( i ) )
$
holds for all $i > 0$, so that
$
 H ^{ 0 } \lb \bP _{ \bC }, \cO _{ \bP _{ \bC } } ( i ) \rb
 \to
 H ^{ 0 } \lb X, \cO _{ X } ( i ) \rb
$
is surjective for all $i \ge 0$ (resp. for $Y$).
Next we replace the Fourier--Mukai kernel of the D-equivalence between $X$ and $Y$
with a complex $E$
consiting of
direct sums of line bundles of the form
$
 \cO _{ X } ( i ) \boxtimes \cO _{ Y } ( i )
$
for some integer $i$. By the surjectivity we just mentioned, the differentials of the perfect complex are represented by matrices whose terms are polynomials over $\bC$.
Thus we see that there exists a subring $R \subset \bC$ of finite type over $\bQ$ such that
\begin{itemize}
\item
$X, Y$ are the base changes by
$
 Q ( R ) \subset \bC
$
of the fibers over the generic point
$
 \Spec Q ( R ) \to \Spec R
$
of the schemes $\cX$ and $\cY$ which are defined in the projective space $\bP _{ R }$ over $R$ by a system of homogeneous polynomials with coefficients in $R$, and

\item
the Fourier--Mukai kernel $E$ is the base change of the restriction to the generic fiber of a perfect complex $\cE$ on
$
 \bP _{ R } \times _{ R } \bP _{ R }
$.
\end{itemize}
Now consider the variety
$
 T = \Spec R
$
in characteristic zero.
Since $X$ and $Y$ are smooth and projective over the generic point of $T$, there exists an open dense subset $ U \subset T $ over which
$\cX$ and $\cY$ are smooth and projective.
Since the Hodge numbers are the same for all the fibers of a smooth projective morphism between varieties in characteristic zero,
we can compare the Hodge numbers of $X$ and $Y$ on the fibers of any closed point of $U$ (see the similar argument in the proof of \cite[Proposition 5.1]{MR2019152}).

Take the the right adjoint kernel
$
 \cE _{ R }
 =
 \cE ^{ \vee } \otimes p _{ \cX } ^{ * } \omega _{ \cX / R } [ \dim X ]
$
of $\cE$ (see, say, \cite[Definition 5.7]{MR2244106}).
Consider the adjunction unit map
\begin{align}
 \mu \colon \cO _{ \Delta _{ \cX / R } } \to \cE _{ R } * \cE,
\end{align}
where
$
 *
$
denotes the convolution of kernels
$
 p _{ \cX, \cX * } \lb p _{ \cY, \cX } ^{ * } \cE _{ R } \otimes p _{ \cX, \cY } ^{ * } \cE \rb
$,
so that
\begin{align}
 \Phi ^{ \cX \to \cX } _{ \cE _{ R } * \cE }
 =
 \Phi ^{ \cY \to \cX } _{ \cE _{ R } } \circ \Phi ^{ \cX \to \cY } _{ \cE }.
\end{align}
Since
$
 \Phi ^{ X \to Y } _{ E }
$
is fully faithful, it follows that the support of the cone of $\mu$ does not intersect the generic fiber
$
 \cX _{ Q ( R ) }
$.
Since $\cX$ is proper over $T$, by removing the image of the support of the cone of $\mu$ if necessary, we may assume that
$
 \mu
$
is an isomorphism over $U$.
The flatness of $\cY$ over $U$ implies that
$
 p _{ \cX, \cX * }
$
commutes with restriction over any point $u \in U$,
%the isomorphism of functors
%$
% | _{ u } \circ p _{ \cX, \cX * }
% \simeq
% p _{ \cX _{ u }, \cX _{ u } * } \circ | _{ u }
%$,
so that one can easily verify
$
 \lb \cE _{ R } * \cE \rb | _{ u }
 \simeq
 \cE _{ R } | _{ u } * \cE | _{ u }
 \simeq
 \lb \cE | _{ u } \rb _{ R } * \cE | _{ u }
$.
Thus we see that
$
 \Phi _{ \cE | _{ u } }
$
is fully faithful for any
$
 u \in U
$.
Combined with
similar arguments for the adjunction counit map, one can deduce that 
$
 \Phi _{ \cE | _{ u } }
$
is an equivalence for any
$
 u \in U
$
after removing some non-dense closed subset from $U$ if necessary.
Therefore we can take the number field $K$ to be the residue field of any closed point of $U$.
\end{step}

\begin{step}
Once everything is defined over a number field $K$, then $X$, $Y$ extend to projective schemes over the ring of integers
$
 \cO _{ K } \subset K
$
and the Fourier--Mukai kernel $E$ extend to a perfect complex on their fiber product over $\cO _{ K }$
(recall that we may assume that $E$ is a perfect complex whose terms are direct sums of line bundles of the form
$
 \cO _{ X } ( i ) \boxtimes \cO _{ Y } ( i )
$
and the differentials are matrices whose entries are polynomials with coefficients in $K$).
Moreover, by similar arguments as above, one can find an open dense subset
$
 V \subset \Spec \cO _{ K }
$
over which $X$ and $Y$ extend to smooth schemes and also the Fourier--Mukai kernel $E$ yields D-equivalence between any pair of fibers.
Now, if one assumes the conjecture that D-equivalence implies the coincidence of Hasse--Weil zeta functions over finite fields, one can apply it to the reductions of $X$ and $Y$ to the closed points of $V$.
This then would imply the coincidence of the Hodge numbers in characteristic zero by \cite[Proposition 1.2]{MR2019152}.
\end{step}
\end{proof}

Now we consider the possible modifications of \pref{pb:naive}.
In the study of arithmetic and geometry of abelian varieties,
it is often convenient to work
with the category of abelian varieties up to isogeny.
Let
$
 K ' _{ 0 } \lb \Var / \bfk \rb
$
be the quotient of
$
 K _{ 0 } \lb \Var / \bfk \rb
$
by the ideal generated by
$
 [ A ] - [ B ],
$
where $A$ and $B$ are
isogenous abelian varieties.
The counting measure factors through
$
 K ' _{ 0 } \lb \Var / \bF _{ q } \rb [ \bL ^{ - 1 } ],
$
and the Hodge--Deligne polynomial factors through
the composition with the completion
$
 K ' _{ 0 } \lb \Var / \bC \rb [ \bL ^{ - 1 } ] \to \Khat ' _{ 0 } \lb \Var / \bC \rb
$
with respect to the filtration
induced by that of
$
 K _{ 0 } \lb \Var / \bfk \rb
$
in \pref{df:Grothendieck_ring}.
In view of the result by Orlov \cite{MR1921811}
that D-equivalent abelian varieties are isogenous,
it is natural to ask the following:

\begin{problem} \label{pb:modulo_isogeny}
Let $(X, Y)$ be a Fourier--Mukai pair.
Does the equality $[X] = [Y]$ hold either in 
$
 K ' _{ 0 } \lb \Var / \bfk \rb [ \bL ^{ - 1 } ]
$
or
%its completion
$
 \Khat ' _{ 0 } \lb \Var / \bfk \rb?
$
%when $X$ and $Y$ are D-equivalent?
\end{problem}

%It is interesting to ask if this tentative modification of \pref{pb:naive}
%works for \pref{cr:nonDL2}.

One could instead consider the Grothendieck ring
$
 K _{ 0 } \lb \CM _{ \bfk } \rb
$
of the
$
 \bQ
$-linear tensor category
$
 \CM _{ \bfk }
$
of Chow motives with rational coefficients over $\bfk$.
%where the rational equivalence is considered.
%As usual,
The Chow motive of a smooth projective variety $X$ will be denoted by $\frakh ( X )$.

%We raise the following conjecture.

\begin{problem}\label{pb:chow_version}
Let $(X, Y)$ be a Fourier--Mukai pair.
Does the equality
$
 \ld \frakh(X) \rd = \ld \frakh(Y) \rd
$
hold in
$
 K _{ 0 } \lb \CM _{ \bfk } \rb?
$
\end{problem}

A conjecture of Orlov \cite[Conjecture 1]{MR2225203},
which asserts that the effective numerical Chow motives of
%in $\CM_{\bfk}$
any Fourier--Mukai pair $(X,Y)$ are isomorphic,
implies the affirmative answer to \pref{pb:chow_version}.
Here we note that the conjecture has been confirmed for some of the examples mentioned in the introduction, in a series of works by Laterveer \cite{2018arXiv180808338L,2019arXiv190104812L,MR3819748}. 

Orlov's conjecture is partly motivated by the fact,
proved in \cite{MR2225203, MR2196100, MR2641191, MR3108695},
that
for any Fourier--Mukai pair $(X,Y)$,
the images of $\frakh(X)$ and $\frakh(Y)$
in the orbit category $\CM_{\bfk} / \bZ^{\bQ(1) \otimes (-)}$
of $\CM_{\bfk}$
with respect to the Tate twist functor $\bQ(1) \otimes (-)$
are isomorphic.
This roughly means that we can recover the motive $\frakh(X)$
of a smooth projective variety $X$
from the derived category $D(X)$
up to the information of codimensions of cycles.
%(or weights),
\pref{pb:chow_version} asks if this remaining information can also be recovered.
For the case \( \bfk = \bF _{ q } \) one should note that an affirmative answer to \pref{pb:chow_version} only implies the coincidence of the number of rational points modulo \( q \) (see the paragraph after \pref{pb:numerical_version} below), whereas an affirmative answer to \pref{pb:modulo_isogeny} implies the actual coincidence.

When $\bfk$ is of characteristic zero, there is a motivic measure
\begin{align}\label{eq:GS}
 \mu _{ \GS } \colon K _{ 0 } \lb \Var / \bfk \rb [ \bL ^{ - 1 } ] \to K _{ 0 } \lb \CM _{ \bfk } \rb
\end{align}
due to \cite{MR1409056"ü"¿'Ì'£i}
which sends the class
$
 [ X ]
$
of a smooth proper variety $X$ to
$
 [ \frakh ( X ) ]
$.
Since isogenous abelian varieties have isomorphic Chow motives
(cf.~e.g.~ \cite[Theorem 2.7.2 (c)]{MR3052734}),
it factors through a ring homomorphism
$
 K ' _{ 0 } \lb \Var / \bfk \rb [ \bL ^{ - 1 } ] \to K _{ 0 } \lb \CM _{ \bfk } \rb
$.
Hence, in characteristic 0, the affirmative answer to \pref{pb:modulo_isogeny}
implies that to \pref{pb:chow_version}.

\begin{remark}\label{rm:question_of_huybrechts}
The difference
$
 \ld A \rd - \big[ \Ahat \, \big] \in K _{ 0 } \lb \Var / \bC \rb \ld \bL ^{ - 1 } \rd
$
for an abelian variety $A$ as in \pref{lm:existence} is a non-trivial element of the kernel of
$
 \mu _{ \GS }
$,
since $A$ and $\Ahat$ are isogenous.
\end{remark}

One can also consider the category $\NM_\bfk$
of numerical motives with rational coefficients over $\bfk$.
The numerical motive of a smooth projective variety $X$
will be denoted by $\frakn(X)$.
The numerical version of \pref{pb:chow_version} is the following:

\begin{problem} \label{pb:numerical_version}
Let $(X, Y)$ be a Fourier--Mukai pair.
Does the equality $\ld \frakn(X) \rd = \ld \frakn(Y) \rd$ hold in
$
 K _{ 0 } \lb \NM _{ \bfk } \rb?
$
\end{problem}

The affirmative answer to \pref{pb:chow_version} implies that to \pref{pb:numerical_version}.
Unlike the case of Chow motives,
the equality
$
 \ld \frakn(X) \rd = \ld \frakn(Y) \rd
$
in
$
 K _{ 0 } \lb \NM _{ \bfk } \rb
$
is close to the equality
$
 \frakn(X) = \frakn(Y)
$
in
$
 \NM _{ \bfk },
$
since
$
 \NM _{ \bfk }
$
is semisimple abelian
\cite{MR1150598}.
%Although much weaker than \pref{pb:chow_version},
The affirmative answer to \pref{pb:numerical_version}
for a Fourier--Mukai pair $(X,Y)$ over the finite field \( \bF _{ q } \) implies the equality of the number of rational points modulo \( q \) (see \cite[Theorem 9.1]{MR2520468}).
Similarly, modulo the standard conjectures, the affirmative answer to \pref{pb:numerical_version} for a Fourier--Mukai pair $(X,Y)$ over $\bC$ implies the equality of Hodge numbers of $X$ and $Y$ (see \cite[Section 7.1.2]{MR2115000}).

To conclude the paper,
we give an affirmative answer to \pref{pb:chow_version} (and hence to \pref{pb:numerical_version})
for the pairs discussed in \pref{cr:nonDL2}.

\begin{proposition}
If
$
 ( A, B )
$
is a pair of D-equivalent complex abelian surfaces,
then one has
$
 [ \frakh ( A ^{ [ n ] } ) ] = [ \frakh ( B ^{ [ n ] } ) ] \in K _{ 0 } \lb \CM _{ \bC } \rb
$
for any positive integer $n$.
\end{proposition}

The following argument is a slight modification of the proof of \cite[Lemma 2.1]{2018arXiv180109385O}.
\begin{proof}
As we quoted above, Orlov proved that $A$ and $B$ are isogenous and hence
$
 \frakh ( A ) \simeq \frakh ( B )
$.
To show the case
$
 n \ge 2
$,
let us consider the generating series as follows.
\begin{align}
 \bH _{ Z } ( T ) \coloneqq \sum _{ n = 0 } ^{ \infty } \ld Z ^{ [ n ] } \rd T ^{ n }
 = 1 + [ Z ] T + \cdots
 \in
 1 + T \cdot K _{ 0 } \lb \Var / \bC \rb \ldd T \rdd.
\end{align}

By \cite[COROLLARY]{gusein-zade2006}, for any smooth quasi-projective variety
$
 Z,
$
there is an equality
\begin{align}\label{eq:eq_in_Grothendieck_ring}
 \bH _{ Z } ( T ) =
 \lb \bH _{ \bA ^{ \dim Z } } ( T ) \rb ^{ \bL ^{ - \dim Z } \ld Z \rd }
 \in 1 + T \cdot K _{ 0 } \lb \Var / \bC \rb \ld \bL ^{ - 1 } \rd \ldd T \rdd
\end{align}
(see \cite{gusein-zade2006} and references therein for the notion of the power structure on
$
 K _{ 0 } \lb \Var / \bfk \rb
$).
On the other hand, since the Gillet--Soul\'e measure \eqref{eq:GS}
preserves the $\lambda$-ring structures (see \cite[Section 4.3]{MR2377891} and the proof of \cite[Proposition 4.8]{MR3437198}) and hence the power structures (see \cite[Proposition 1]{gusein-zade2006}), 
\eqref{eq:eq_in_Grothendieck_ring} implies the similar formula in the Grothendieck ring of Chow motives
\begin{align}
 \mu _{ \GS } \lb \bH _{ Z } ( T ) \rb
 &=
 \sum _{ n = 0 } ^{ \infty } \ld \frakh ( Z ^{ [ n ] } ) \rd T ^{ n } \\
 &=
 \mu _{ \GS } \lb \bH _{ \bA ^{ \dim Z } } ( T ) \rb ^{ \ld \frakh ( Z ) ( - \dim Z ) \rd }
 \in
 1 + T \cdot K _{ 0 } \lb \CM _{ \bC } \rb \ldd T \rdd,
\end{align}
where
$
 ( - )
$
is the Tate twist.
By comparing the coefficients of
$
 T ^{ n }
$
in the equality
\begin{align}
 \mu _{ \GS } \lb \bH _{ A } ( T ) \rb
 &=
 \mu _{ \GS } \lb \bH _{ \bA ^{ \dim A } } ( T ) \rb ^{ \ld \frakh ( A ) ( - \dim A ) \rd } \\
 &=
 \mu _{ \GS } \lb \bH _{ \bA ^{ \dim B } } ( T ) \rb ^{ \ld \frakh ( B ) ( - \dim B ) \rd } \\
 &=
 \mu _{ \GS } \lb \bH _{ B } ( T ) \rb,
\end{align}
we obtain the assertion.
\end{proof}

%\begin{remark}
%It is not clear to the authors if the power structure on
%$
% K _{ 0 } \lb \Var / \bfk \rb [ \bL ^{ - 1 } ]
%$
%descends to
%$
% K ' _{ 0 } \lb \Var / \bfk \rb [ \bL ^{ - 1 } ]
%$
%or not.
%
%The authors do not know how to prove the isomorphism
%$
% \frakh ( A ^{ [ n ] } ) \simeq \frakh ( B ^{ [ n ] } )
%$
%for
%$
% ( A, B )
%$
%as in the previous proposition.
%\end{remark}

%Let further
%$
% \Khat ' _{ 0 } \lb \Var / \bfk \rb
%$
%be the completion of the localization
%$
% K ' _{ 0 } \lb \Var / \bfk \rb [ \bL ^{ - 1 } ]
%$
%by the obvious filtration.
%One can easily check that
%
%It is an interesting question
%whether the classes of
%$A^{[n]}$ and $\Ahat^{[n]}$
%appearing in \pref{cr:nonDL2}
%are equal in $\Khat ' _{ 0 } \lb \Var / \bC \rb$.

%%
%\item
%How about the generalized Kummer varieties
%$
% K _{ n } \lb A \rb
%$
%and
%$
% K _{ n } \lb \Ahat \rb
%$,
%which are the fibers of the origins of the summation maps
%$
% s
%$
%and
%$
% \shat
%$?
%\end{enumerate}

\bibliographystyle{amsalpha}
%\bibliography{mainbibs}
\bibliography{Ann2_DLAV}

\end{document}